\documentclass[12pt]{article}
\usepackage[utf8]{inputenc}
\usepackage[IL2]{fontenc}
\usepackage[czech]{babel}
\usepackage{amssymb}
\usepackage{amsmath} 
\usepackage{amsthm} 
\usepackage{amssymb} 
\usepackage{commath}
\usepackage{amsfonts} 
\usepackage{enumerate}
\usepackage{subcaption}

\newtheorem{theorem}{Theorem}

\newtheorem{lemma}[theorem]{Lemma}
\newtheorem{remark}[theorem]{Remark}
\theoremstyle{definition}
\newtheorem{defi}{Definition}

\usepackage{tikz}
\usetikzlibrary{calc,decorations}
\pgfdeclaredecoration{arrows}{draw}{
\state{draw}[width = \pgfdecoratedinputsegmentlength]{%
  \path [every arrow subpath/.try] \pgfextra{%
    \pgfpathmoveto{\pgfpointdecoratedinputsegmentfirst}%
    \pgfpathlineto{\pgfpointdecoratedinputsegmentlast}%
   };
}}
\tikzset{
                 base/.style = { circle, draw },
               filled/.style = { base, fill = black!50 },
  every arrow subpath/.style = { ->, draw, thick }
}

\begin{document}
\renewcommand{\proofname}{Proof}
\renewcommand{\figurename}{Figure}

\title{\textit{Remarks on definitions of periodic points for nonautonomous dynamical system}}
\date{December 9, 2018}
\author{Vojt\v{e}ch Pravec}

\maketitle
\noindent\textbf{Abstract}
\\

\noindent Let $(X,f_{1,\infty})$ be a nonautonomous dynamical system. In this paper we summarize known definitions of periodic points for general nonautonomous dynamical systems and propose a new definition of asymptotic periodicity. This definition is not only very natural but also resistant to changes of a beginning of the sequence generating the nonautonomous system. We show the relations among these definitions and discuss their properties. We prove that for pointwise convergent nonautonomous systems topological transitivity together with dense set of asymptotically periodic points imply sensitivity. We also show that even for uniformly convergent systems the nonautonomous analog of Sharkovsky’s Theorem is not valid for most definitions of periodic points.
\\

\\


\section{Introduction}
Denote by $\mathbb{N}$ the set of positive integers. Let $X$ be a compact metric space and $f_{1,\infty}:=\{f_i\}_{i=1}^\infty$ a sequence of continuous maps $f_i:X\rightarrow X$. By a \textit{nonautonomous dynamical system} (NDS, for short) we mean a pair $(X,f_{1,\infty})$. Denote by $\mathcal{F}(X)$ the class of such systems. For $(X,f_{1,\infty})$ and any $i,n\in \mathbb{N}$ put $f_i^0=id_X$ and $f_i^n=f_{i+(n-1)}\circ \cdots \circ f_{i+1}\circ f_i$, the $n$-th iteration of $f_i$. The \textit{trajectory} of a point $x\in X$ is the sequence $\{f_1^n(x)\}_{n=0}^\infty$ and the \textit{orbit} of a point $x\in X$ is the set $\mathcal{O}(x,f_{1,\infty})=\{x,f_1^1(x),f_1^2(x),\dots,f_1^n(x),\dots\}$. Note that a “classical” autonomous dynamical system $(X,f)$ is a special case of NDS with $f_n=f$ for any $n\in\mathbb{N}$. We assume that $f_n$ is surjective for any $n\in\mathbb{N}$ since we want to avoid some pathological examples (for instance if $f_1$ is constant then the whole dynamics of $(X,f_{,\infty})$ is shrinked to a single trajectory).

In a part of this paper we study properties related to periodicity for a special class $\mathcal{F}_0(X)$ of uniformly convergent NDS, i.e., such that the sequence $f_{1,\infty}$ converges uniformly to a continuous map $f$. Note that such systems are currently studied intensively - see, e.g. \cite{12}-\cite{9}, \cite{3}-\cite{11}.

The notion of a periodic point in the case of autonomous system $(X,f)$  (\textit{x is periodic with period} $n$ if there is $n\in\mathbb{N}$ such that $f^n(x)=x$) is very natural and intuitive. But if we consider a NDS then it is more difficult to decide what a good definiton of a periodic point should be. In recent years, many papers appeared, where different definitions of periodicity for NDS have been introduced. In the next subsection we will give a survey of them and propose a new definition. Note that all definitions of periodicity can be stated for general NDS without any assumptions.

\subsection{Definitions of periodic points}
Let $X$ be a metric space and $f_{1,\infty}$ a sequence of continuous maps of $X$
\begin{defi}\label{def1} (Cánovas, 2011 \cite{2})

A point $x\in X$ is {\it pseudo-periodic} of $f_{1,\infty}$ if there exists $r\in \mathbb{N}$ and points $x_1=x,x_2,\dots,x_r$ such that for any $\varepsilon >0$ there is $N\in \mathbb{N}$ such that $d(f_{nr+i}(x_i),x_{i+1})<\varepsilon$ for any $n\geq N, 1\leq i\leq r,$ and $d(f_{nr}(x_r),x_1)<\varepsilon$ for $n\geq N$. The smallest natural number $r$ satisfying the above conditions is called the period of $x$. Denote by $PS(f_{1,\infty})$ the set of pseudo-periodic points of $f_{1,\infty}$.
\end{defi}
 \begin{defi}\label{def2}(Shi and Chen, 2009 \cite{5})

A point $x\in X$ is called \textit{$r$-periodic} if $f_1^{n+r}(x)=f_1^n(x),n\geq 0$.
\end{defi}
 \begin{defi}\label{def3}(Sánchez, Sanchiz and Villanueva, 2017 \cite{4})

A point $x\in X$ is \textit{periodic} if $f_1^n(x)=x$ for some positive integer $n$.
\end{defi}
 \begin{defi}\label{def4}(Miralles, Murillo-Arcila and Sanchiz, 2018 \cite{3})

A point $x\in X$ is \textit{$r$-periodic} if $f_1^{rn}(x)=x$ for  any $n\in \mathbb{N}$.
 \end{defi}

All of the above mentioned definitions are generalizations of periodicity for ADS, i.e., if a NDS is such that $f_n=f$ for any $n\in\mathbb{N}$, then periodicity in the sense of Definitions \ref{def1}-\ref{def4} coincides with the standard periodicity of ADS.

But if we study dynamics of some system (and it does not matter whether it is autonomous or nonautonomous), we are in fact interested in a long term behavior of trajectories of points. Therefore, in the case of NDS, changing or deleting of finitely many functions should not affect possible periodicity of a given trajectory. Note that Definitons \ref{def1}-\ref{def4} do not satisfy this natural requirement. For example if we consider a NDS such that $f_1=id_X$ and $f_2=f_3=\dots=\tau$, where $\tau$ is the {\it tent map} ($\tau:I\rightarrow I, \tau(x)=1-\abs{1-2x}$), then every point is periodic in the sense of Definiton \ref{def3}, despite the fact that there are points with dense orbits. Hence Definition \ref{def3} seems to be too benevolent. On the other hand  Definitions~\ref{def2} and \ref{def4} are too strict since they do not allow even a small change in the trajectory of a periodic point if we consider the previous example, then the only periodic points are fixed points.

Definition \ref{def1} works with pseudo-periodicity which is more suitable for a NDS than strict periodicity. But it defines the term “a point is pseudo-periodic” and does not take into account the trajectory of the point itself. One can easily construct a NDS with points whose trajectories do not behave periodically but they are periodic in the sense of Definition \ref{def1} (e.g., see the proof of Theorem \ref{relations}, part iii)).

Since none of the above mentioned definitions does not fulfill our conception and requirements for periodicity in a NDS, we propose a new definition.

\begin{defi}\label{def5} A point $x\in X$ is {\it asymptotically periodic} with the cycle $x_1,x_2\dots,x_r$ if there exist $x_1,x_2,\dots,x_r\in X$ such that for any $\varepsilon>0$ there is $n_0>0$ such that 
$d(f_1^{nr+i}(x),x_{i+1})<\varepsilon$ for $n\geq n_0$ and $0\leq i<r$.
Denote by $AP(f_{1,\infty})$ the set of asymptotically periodic points of $f_{1,\infty}$.
\end{defi}
Hence, a point is periodic in the sense of this definition if it eventually approaches some cycle. This definition is in fact closer to \textit{eventual periodicity} than \textit{periodicity} in the case of ADS, but in fact, this approach is for NDS more natural. Note also that, if a point is asymptotically periodic then its omega limit set $\omega_{f_{1,\infty}}(x)$ is finite but in general the converse implication is not true.

The following theorem summarizes the relations among the definitions mentioned above.
\begin{theorem}\label{relations}

 The relations among Definitions 1-5 are as follows:
\begin{figure}[h!]
\begin{center}
\begin{subfigure}[b]{0.4\textwidth}
\setlength{\unitlength}{4.5cm}
\begin{tikzpicture}
  \draw node[base] (1)  at (1,1)								{1}
 		 node[base] (2) at(2,2)                             {2}
 		  node[base] (3) at(4,-0) 							{3}
 		   node[base] (4) at(3,1) 							{4}
 		    node[base] (5) at(2,-0) 							{5}

  ;
  
  \draw[decoration=arrows, decorate]
      (2)--(4)
      (2)--(1)
      (4)--(3)
      (2)--(5)

;

\end{tikzpicture}
\caption{General NDS}
\label{gliche}
\end{subfigure}
\begin{subfigure}[b]{0.4\textwidth}
\setlength{\unitlength}{4.5cm}\begin{tikzpicture}
   \draw node[base] (1)  at (1,1)								{1}
 		 node[base] (2) at(2,2)                             {2}
 		  node[base] (3) at(4,-0) 							{3}
 		   node[base] (4) at(3,1) 							{4}
 		    node[base] (5) at(2,-0) 							{5}

  ;
  
  \draw[decoration=arrows, decorate]
      (2)--(4)
      (2)--(1)
      (4)--(3)
      (4)--(5)
	  (4)--(1)	
     
;

\end{tikzpicture}
\caption{Uniformly converging NDS}
\label{gsude}
\end{subfigure}
\end{center}
\end{figure}

An arrow in the graph represents the implication between the corresponding definitions and the missing arrow means that the implication is not true except for those following by transitivity.
 \end{theorem}
\begin{proof}
\begin{enumerate}
\item[a)] Case of general NDS 

\begin{enumerate}
\item[i)] From the definitions we immediately obtain the following implications:
		$2\implies 1$,  $2\implies 4$, $2\implies 5$ and $4\implies 3$.
\item[ii)] Now we show that Definition \ref{def3} does not imply any other definition.
	Fix $x_0\in X$ such that $x_0$ is a transitive point for $\tau$. Let $f_{1,\infty}$ be a nonautonomous system such that $f_1=id$ and $f_2=f_3=\dots=\tau$.Then $x_0$ is periodic in the sense of Definition \ref{def3} but not periodic in the sense of any other definition since its orbit is dense.
	\item[iii)] Consider the following nonautonomous system $f_{1,\infty}$.
	Let $y<2/3$ be a point with dense orbit under the tent map. Let $ f_1 $ be such that $f_1(0)=0,\ f_1(y)=y,\ f_1(5/6)=y,\ f_1(1)=1$ and it is linear between these points and let $f_n=\tau$ for $n>1$. Then the point $x_0=2/3$ is periodic in the sense of Definition \ref{def1} but not in the sense of any other definitions. 
    \item[iv)] Consider nonautonomous system from iii), but with $y=2/3$. Then there exists a point $2/3<x_0<5/6$ with dense orbit under the tent map. Such a point $x_0$ is periodic in the sense of Definition \ref{def5} but not in the sense of any other definition.
    \item [v)] Let $q_1,q_2,q_3,\dots$ be an enumaration of rational numbers in $[0,1]$. Consider the following NDS $f_{1,\infty}$. Let $f_n$ be a surjective interval map for $n\in \mathbb{N}$ such that $f_{2m}(q_m)=0$ and $f_{2m-1}(0)=q_m$ for $m\in\mathbb{N}$. Then the point $x_0=0$ is periodic in the sense of Definition \ref{def4} but not in the sense of Definitions \ref{def1}, \ref{def2} and \ref{def5}.
    
    \end{enumerate} 
    \item[b)] Case of uniformly converging NDS
    
    Note that parts i)-iv) from the general case are also true for uniformly converging NDS. Moreover, from Definitions \ref{def1}, \ref{def4} and \ref{def5} together with uniform convergence we easily obtain that: $4\implies 5$ and $4\implies 1$ 
    \end{enumerate}
\end{proof}

Cánovas in \cite{2} proved for NDS $(X,f_{1,\infty})$ uniformly converging to $f$ that a point $x$ is pseudo-periodic (Definition \ref{def1}) if and only if $x$ is a periodic point for $f$. For asymptotically periodic points (Definition \ref{def5}) we have the following lemma.
\begin{lemma}\label{lemma}
Let $f_{1,\infty}$ be a nonautonomous system. Suppose that $f_n$ converges uniformly to $f$. If $x\in AP(f_{1,\infty})$ with the cycle $x_1,x_2,\dots,x_r$ then $x_1,x_2,\dots,x_r$ is a periodic cycle for $f$.
\end{lemma}
\begin{proof}
Let $\varepsilon>0$. Let $N_1$ be such that $d(f_{nr+1}^r(x),f^r(x))<\varepsilon/2$ for all $n\geq N_1$, $\delta$ be such that $d(x,y)<\delta \implies d(f^r(x),f^r(y))<\varepsilon/2$ and let $N_2$ be such that $d(f_1^{nr}(x),x_1)<\delta$ for all $n\geq N_2$. Set $N=\max\{N_1,N_2\}$.

Let $n\geq N$. We get 
\begin{align*}
 d(f_1^{(n+1)r}(x),f^r(x_1))&=d(f_{nr+1}^r(f_1^{nr}(x)),f^r(x_1))\leq \max_{y\in B_{\delta}(x_1)}d(f_{nr+1}^r(y),f^r(x_1))\\ &
\leq \max_{y\in B_{\delta}(x_1)}(d(f_{nr+1}^r(y),f^r(y))+d(f^r(y),f^r(x_1)))<\varepsilon
\end{align*}
We get that $\lim_{n\rightarrow\infty}f_{nr+1}^r(f_1^{nr}(x))=f^r(x_1)$, but $x\in AP(f_{1,\infty})$ and from the definition we can see that $\lim_{n\rightarrow\infty}f_{nr+1}^r(f_1^{nr}(x))=x_1$, therefore $f^r(x_1)=x_1$ and $x_1$ is periodic point for $f$. Analogously for $x_2, x_3, \dots,x_r$.
\end{proof}
\begin{remark}Note that this lemma is also true for Definitons \ref{def2}, \ref{def4} but is not true for Definition \ref{def3}.
\end{remark}
\begin{remark}
The converse implication is not true for Definitions \ref{def2}-\ref{def5}: there exists a point $x$ such that $x\in Per(f)$, but $x$ is not periodic in the sense of Definitions \ref{def2}-\ref{def5} (see example in the proof of Theorem \ref{sharkovsky}).
\end{remark}
\section{Transitivity and Devaney chaos}
Let us now discuss the notion of Devaney chaos for NDS. For autonomous systems generated by a continuous map $f$ of a metric space $X$ there are two commonly used, but in general nonequivalent, definitions of transitivity of $f$:
\begin{enumerate}
\item[(TT)] For every pair of nonempty open sets $U$ and $V$ in $X$, there is a positive integer $n$ such that $f^n(U)\cap V\neq \emptyset$.
\item[(DO)] There is a point $x_0$ in $X$ such that the orbit of $x$ is dense in $X$.
\end{enumerate}
However, if $X$ is compact and without isolated points then these two definitions are equivalent (see \cite{6}). But even for interval NDS we can easily get that (DO) does not imply (TT) (see, e.g., the example in the proof of Theorem \ref{relations}, part e)). In the following text we will use the following definition.
\begin{defi}
A nonautonomous system $(X,f_{1,\infty})$ is {\it topologically transitive} if for every pair of nonempty open sets U and V in X, there is a positive integer $n$ such that $f_1^n(U)\cap V\neq \emptyset$.
\end{defi}
We say that NDS $(X,f_{1,\infty})$ has $\textit{sensitive dependence on initial conditions}$ (is sensitive, for short) if there is $\delta>0$ such that for any $x\in X$ and $\varepsilon>0$ there is $y\in X$ with $d(x,y)<\varepsilon$ such that $d(f_1^n(x),f_1^n(y))>\delta$ for some $n\geq 0$.

Now, let us recall the Devaney's definition of chaos for NDS.
\begin{defi}
A NDS $(X,f_{1,\infty})$ is {\it Devaney chaotic} if it satisfies the following conditions:
\begin{enumerate}
\item[i)] $(X,f_{1,\infty})$ is topologically transitive,
\item[ii)]  $(X,f_{1,\infty})$ has a dense set of periodic points,
\item[iii)]  $(X,f_{1,\infty})$ is sensitive.
\end{enumerate}
\end{defi}

It is known (see \cite{1}) that in the case of ADS the first two conditions imply the third one and for {\it interval} ADS (see \cite{6}) it is true that transitivity implies dense set of periodic points and hence transitivity is equivalent to Devaney chaos. A natural question arises if the analogous properties are valid for NDS (with the appropriate definition of periodic points.)

It is easy to see that the analogous of this theorem are not valid if we consider periodicity in the sense of Definition \ref{def1} or Definition \ref{def3}. Zhu, Shi and Shao proved in \cite{11} that the first two conditions imply the third one if we consider periodicity in the sense of Definiton \ref{def2}. Miralles, Murillo-Arcila and Sanchis proved in \cite{3} analogous theorem for Definition \ref{def4}. 

The following theorem shows that this property remains to be fulfilled even for our new weaker definition of asymptotical periodicity. Note that in general  it is not true that in the case of \textit{interval} NDS transitivity implies existence of  dense set of periodic points in the sense of Definitions \ref{def2}-\ref{def5} (see example in the proof of Theorem \ref{relations}, part iii)).

\begin{theorem}\label{devaney}
Let $f_{1,\infty}$ be a nonautonomous system on a compact metric space X without isolated points. Suppose that $f_n$ converges uniformly to $f\in C(X)$. If
\begin{enumerate}
\item $f_{1,\infty}$ is topologically transitive and
\item $AP(f_{1,\infty})$ is dense 
\end{enumerate}
then $f_{1,\infty}$ is sensitive.
\end{theorem}
\begin{proof}
First, assume that all asymptotically periodic points have the same cycle $x_1,x_2,\dots,x_r$. Then there is a point $x_0\in X$ and $\gamma>0$ such that $d(x_0,x_i)>\gamma$ for $i=1,2,\dots,r$. We claim that $f_{1,\infty}$ is sensitive with the constant $\delta:=\gamma/4$.

Let $x\in X$ and $\varepsilon>0$. Then there is $q\in AP(f_{1,\infty})$ and $N\in \mathbb{N}$ such that $d(x,q)<\varepsilon$ and $d(f_1^{nr+i}(q),x_{i+1})<\delta$ for any $n\geq N$. Next, from the fact that $f_{1,\infty}$ is topologically transitive, there exist infinitely many positive integers $l$ such that $f_1^l(B_{\varepsilon}(x)\cap B_{\delta}(x_0))\neq \emptyset$ (\cite{3}), thus we can take $k=jr+i$ such that $j>N$ and $f_1^k(B_{\varepsilon}(x)\cap B_{\delta}(x_0))\neq \emptyset$. Then there exists $z\in B_{\varepsilon}(x)$ such that $f_1^k(z)\in B_{\delta}(x_0)$. Using the triangle inequality we obtain:
$$
d(f_1^{jr+i}(z),f_1^{jr+i}(q))\geq d(x_0,x_{i+1})-d(f_1^{jr+i}(z),x_0)-d(f_1^{jr+i}(q),x_{i+1})
$$
$$>4\delta -\delta -\delta$$
Hence, either $d(f_1^{jr+i}(x),f_1^{jr+i}(q))>\delta$ or $(f_1^{jr+i}(x),f_1^{jr+i}(z))>\delta$ which completes the proof.

Now, assume that there are at least two points $x,y\in AP(f_{1,\infty})$ which have disjoint cycles. Let $a,b$ be two arbitrary periodic points for $f$ with disjoint orbits. Let $\gamma$ be the distance between their orbits i.e., $$\gamma:=\min_{m,n\in \mathbb{N}}(f^n(a),f^m(b))$$
 We claim that $f_{1,\infty}$ is sensitive with the constant $\delta:=\gamma/10$.

Let $x\in X$ and $\varepsilon>0$. Since $AP(f_{1,\infty})$ is dense we can find asymptotically periodic point $q$ with the cycle $q_1,q_2,\dots,q_r$ such that $d(x,q)<\varepsilon$. Then there is $i\in \{0,1,\dots,r-1\}$ such that $\forall n\in\mathbb{N}$ $$d(f^n(a),q_{i+1})>5\delta \quad or\quad d(f^n(b),q_{i+1})>5\delta$$
Without loss of generality we can assume that the first case is true. Since $q\in AP(f_{1,\infty})$ there is $j_0\in \mathbb{N}$ such that $$d(f_1^{jr+i}(q),q_{i+1})<\delta\quad \forall j\geq j_0.$$
Since $f$ is continuous on $X$ there is a neighborhood $A$ of the point $a$ such that 
\begin{equation}\label{Devay1}
d(f^l(a),f^l(y))<\delta\quad\text{for}\quad y\in A\quad \text{and}\quad l=0,1,\dots,r.
\end{equation}
Moreover we assume that $f_n$ converges to $f$ (also $f_n^k$ converges to $f^k$) therefore there is $j_1\in \mathbb{N}$ such that $\forall j\geq j_1$, $\forall y\in A$ and $0\leq M<r$ we get
\begin{equation}\label{Devay2}
d(f^{r-M}_{(j-1)r+M+i+1}(y),f^{r-M}(y))<\delta\quad .
\end{equation}
Next, we can take $k=(j-1)r+M+i$, for $0\leq M<r$ such that  $j\geq \max\{j_0,j_1\}$ and $f_1^k(B_\varepsilon(x)\cap A)\neq \emptyset$. Then there is $z\in B_\varepsilon(x)$ such that $f_1^k(z)\in A$.
Using the triangle inequality and inequalities (\ref{Devay1}) and (\ref{Devay2}) we obtain :

$$d(f^{r-M}(a),f^{r-M}_{(j-1)r+M+i+1}(f_1^k(z)))\leq d(f^{r-M}(a),f^{r-M}(f^k_1(z)))+$$

$$d(f^{r-M}(f^k_1(z)),f^{r-M}_{(j-1)r+M+i+1}(f^k_1(z)))<2\delta$$
and then
$$d(f_1^{jr+i}(q),f_1^{jr+i}(z))\geq d(q_{i+1},f^{r-M}(a))-d(q_{i+1},f_1^{jr+i}(q))-d(f_1^{jr+i}(z),f^{r-M}(a))>$$
$$> 5\delta-\delta-2\delta>2\delta.$$
Hence, either $d(f_1^{jr+i}(x),f_1^{jr+i}(q))>\delta$ or $d(f_1^{jr+i}(x),f_1^{jr+i}(z))>\delta$ which completes the proof.
\end{proof}

\section{Sharkovsky’s Theorem}
One of the most famous and important theorems of the theory of discrete dynamical systems on the interval which concerns periodic points is the following Sharkovsky’s theorem:

Order the natural numbers as follows:
\begin{align*}
3\succ 5 \succ 7\succ 9\succ 11\succ\dots\succ 2\cdot3\succ 2\cdot 5\succ 2\cdot 7 \succ\dots\succ\\
2^n\cdot 3\succ 2^n\cdot 5\succ 2^n\cdot 7\succ\dots\succ 2^n\succ\dots\succ 2^2\succ 2\succ 1
\end{align*}
\begin{theorem}[Sharkovsky’s Theorem]
If $f\in C(I)$ has a periodic point with period l and $l\succ m$, then f has a periodic point with period m. 
\end{theorem}

Cánovas in \cite{2} proved that the analogue of the Sharkovsky’s theorem is valid for uniformly converging NDS if we define periodic points using Definition \ref{def1}. 
The following theorem shows that similar theorem is no longer true if we consider periodic points in the sense of Definition \ref{def5}.
\begin{theorem}\label{sharkovsky}

There exists a nonautonomous system $(I,f_{1,\infty})$ which uniformly converges to $f(x):=1-x$ and  with asymptotic 2-cycle but without any asymptotic fixed point.
\end{theorem}
\begin{proof}

    Let $\{g_n\}_{n=1}^\infty$ be a sequence of interval maps such that (for $g_1$ nad $g_2$ see Figure \ref{gn}):
    \begin{itemize}
\item[] if $n$ is odd then 
$$g_n = \left\{ \begin{array}{l@{\quad}c}
    -\frac{n+1}{n+2}x+1, & x \in \left[0,\frac{1}{n+2}\right) \\
    -\frac{x}{(n+1)(n+2)}+\frac{n^2+2n+2}{(n+1)(n+2)}, & x \in \left[\frac{1}{n+2},\frac{2}{n+3}\right) \\
     -\frac{n+2}{n+1}(x-1), & x \in \left[\frac{2}{n+3},1\right] \\
     \end{array} \right. 
    $$  
    \item[] if $n$ is even then
  $$g_n = \left\{ \begin{array}{l@{\quad}c}
    -\frac{n+2}{n+1}x+1, & x \in  \left[0,\frac{n+1}{n+3}\right) \\
    -\frac{x}{(n+1)(n+2)}+\frac{1}{n+2}, & x \in  \left[\frac{n+1}{n+3},\frac{n+1}{n+2}\right] \\
     -\frac{n+1}{n+2}(x-1), & x \in  \left[\frac{n+1}{n+2},1\right] \\
     \end{array} \right. 
    $$  
    \end{itemize}
    
\begin{figure}[h!]
\begin{center}
\begin{subfigure}[b]{0.4\textwidth}
\setlength{\unitlength}{4.5cm}\begin{picture}(1,1,1)(0,-0,1)
\put(0,0){\line(1,0){1}}
\put(0,0){\line(0,1){1}}
\put(1,0){\line(0,1){1}}
\put(0,1){\line(1,0){1}}
\put(0,0){\raisebox{-13pt}{\makebox[0pt][c]{$0$}}}
\put(1,0){\raisebox{-13pt}{\makebox[0pt][c]{$1$}}}
\put(0,1){\raisebox{-3pt}{\makebox[-3pt][r]{$1$}}}
\put(0.5,0){\raisebox{-13pt}{\makebox[-3pt][c]{$\frac{1}{2}$}}}
\put(0.33333,0){\raisebox{-13pt}{\makebox[3pt][c]{$\frac{1}{3}$}}}
\multiput(0,1)(0.1,-0.1){10}
{\line(1,-1){0.08}}
\multiput(0.33333,0)(0,0.11){7}
{\line(0,1){0.09}}
\multiput(0.5,0)(0,0.107){7}
{\line(0,1){0.09}}
\thicklines
\qbezier(0,1)(0.33333,0.77777)(0.333,0.77777)
\qbezier(0.333,0.77777)(0.5,0.75)(0.5,0.75)
\qbezier  (0.5,0.75)(1,0)(1,0)
\end{picture}
\caption{The graph of $g_1$}
\label{gliche}
\end{subfigure}
\begin{subfigure}[b]{0.4\textwidth}
\setlength{\unitlength}{4.5cm}\begin{picture}(1,1,1)(0,-0,1)
\put(0,0){\line(1,0){1}}
\put(0,0){\line(0,1){1}}
\put(1,0){\line(0,1){1}}
\put(0,1){\line(1,0){1}}
\put(0,0){\raisebox{-13pt}{\makebox[0pt][c]{$0$}}}
\put(1,0){\raisebox{-13pt}{\makebox[0pt][c]{$1$}}}
\put(0,1){\raisebox{-3pt}{\makebox[-3pt][r]{$1$}}}
\put(0.6,0){\raisebox{-13pt}{\makebox[0pt][c]{$\frac{3}{5}$}}}
\put(0.75,0){\raisebox{-13pt}{\makebox[0pt][c]{$\frac{3}{4}$}}}
\multiput(0,1)(0.1,-0.1){10}
{\line(1,-1){0.08}}
\multiput(0.6,0)(0,0.0666666){3}
{\line(0,1){0.045}}
\multiput(0.75,0)(0,0.0625){3}
{\line(0,1){0.045}}
\thicklines

\qbezier(0,1)(0.6,0.2)(0.6,0.2)
\qbezier(0.6,0.2)(0.75,0.1875)(0.75,0.1875)
\qbezier (0.75,0.1875)(1,0)(1,0)

\end{picture}
\caption{The graph of $g_2$}
\label{gsude}
\end{subfigure}
\end{center}
\caption{Graphs of $g_{1}$ and $g_2$ }
\label{gn}
\end{figure}
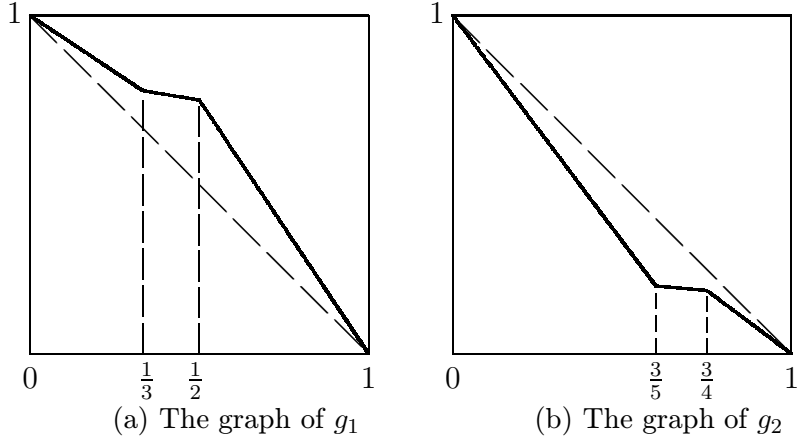
    Next, let $\{F_n\}_{n=1}^\infty$ be a sequence of interval maps such that (see Figure \ref{Fn}):

\begin{itemize}
\item[] if $n$ is odd then 
$$F_n = \left\{ \begin{array}{l@{\quad}c}
    \frac{x}{n+1}, & x \in \left[0,\frac{n+1}{n+2}\right) \\
    (n+1)x-n, & x \in \left[\frac{n+1}{n+2},1\right] \\
     \end{array} \right. 
    $$  
    \item[] if $n$ is even then
    $$F_n = \left\{ \begin{array}{l@{\quad}c}
    -\frac{x}{n+1}+1, & x \in \left[0,\frac{n+1}{n+2}\right) \\
    -(n+1)x+n+1, & x \in \left[\frac{n+1}{n+2},1\right] \\
     \end{array} \right. 
    $$  
    \end{itemize}
    \begin{figure}[h!]
\begin{center}
\begin{subfigure}[b]{0.4\textwidth}
\setlength{\unitlength}{4.5cm}\begin{picture}(1,1,1)(0,-0,1)
\put(0,0){\line(1,0){1}}
\put(0,0){\line(0,1){1}}
\put(1,0){\line(0,1){1}}
\put(0,1){\line(1,0){1}}
\put(0,0){\raisebox{-13pt}{\makebox[0pt][c]{$0$}}}
\put(1,0){\raisebox{-13pt}{\makebox[0pt][c]{$1$}}}
\put(0,1){\raisebox{-3pt}{\makebox[-3pt][r]{$1$}}}
\put(0.8,0){\raisebox{-13pt}{\makebox[0pt][c]{$\frac{n+1}{n+2}$}}}
\multiput(0.8,0)(0,0.07){3}
{\line(0,1){0.05}}
\multiput(0,1)(0.1,-0.1){10}
{\line(1,-1){0.08}}
\thicklines
\qbezier(0,0)(0.8,0.2)(0.8,0.2)
\qbezier(0.8,0.2)(1,1)(1,1)
\end{picture}
\caption{The graph of $F_{n}$ for odd n}
\label{fliche}
\end{subfigure}
\begin{subfigure}[b]{0.4\textwidth}
\setlength{\unitlength}{4.5cm}\begin{picture}(1,1,1)(0,-0,1)
\put(0,0){\line(1,0){1}}
\put(0,0){\line(0,1){1}}
\put(1,0){\line(0,1){1}}
\put(0,1){\line(1,0){1}}
\put(0,0){\raisebox{-13pt}{\makebox[0pt][c]{$0$}}}
\put(1,0){\raisebox{-13pt}{\makebox[0pt][c]{$1$}}}
\put(0,1){\raisebox{-3pt}{\makebox[-3pt][r]{$1$}}}
\put(0.75,0){\raisebox{-13pt}{\makebox[0pt][c]{$\frac{n+1}{n+2}$}}}
\multiput(0.75,0)(0,0.15){5}
{\line(0,1){0.1}}
\multiput(0,0)(0.1,0.1){10}
{\line(1,1){0.08}}

\thicklines

\qbezier(0,1)(0.75,0.75)(0.75,0.75)
\qbezier(0.75,0.75)(1,0)(1,0)

\end{picture}
\caption{The graph of $F_{n}$ for even n}
\label{Fsude}
\end{subfigure}
\end{center}
\caption{Graphs of $F_{n}$  }
\label{Fn}
\end{figure}
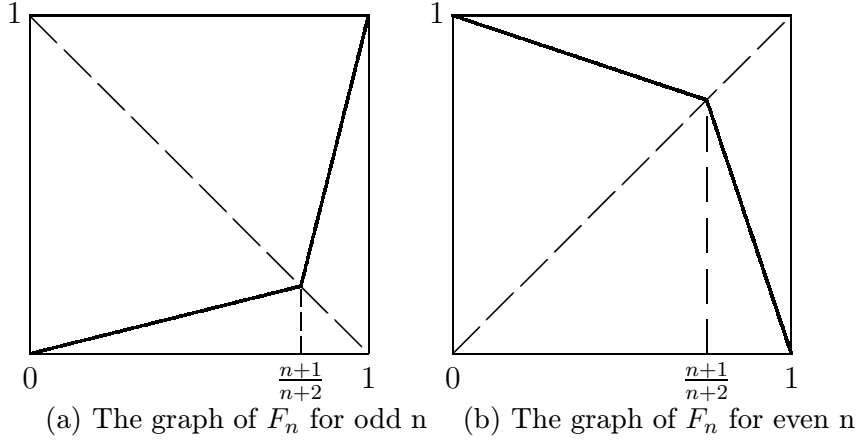

    One can verify that the following equality holds 
   \begin{equation}\label{skladani}   
 g_n(F_n(x))=F_{n+1}(x),\, \forall n\in\mathbb{N}.
     \end{equation} 
    We define $f_{1,\infty}$ in the following way : $f_1=F_1$ and $f_n=g_{n-1}$ for $n>1$.
    
    Now we need to verify that $f_n$ converges uniformly to $f(x):=1-x$. From the definitions of $g_n$ it is easy to see that, for odd $n$, the function $\abs{g_n(x)-(1-x)}$ has maximum either at the point $x_1=1/(n+2)$ or at the point $x_2=2/(n+3)$. Similarly, for even $n$, we have that the maximum is either at the point $x_3=(n+1)/(n+3)$ or at the point $x_4=(n+1)/(n+2)$. One can easily calculate that :
    \begin{itemize}
  \item[] $\abs{g_n(x_1)-(1-x_1)}=\frac{1}{(n+2)^2}$\ \ \
   and\ \ \ $\abs{g_n(x_2)-(1-x_2)}=\frac{1}{(n+3)}$
  \item[] $\abs{g_n(x_3)-(1-x_3)}=\frac{1}{(n+3)}$\ \ \
  and\ \ \ $\abs{g_n(x_4)-(1-x_4)}=\frac{1}{(n+2)^2}$
     \end{itemize}
    Hence for every $\varepsilon>0$ there is $n_0\in\mathbb{N}$ such that $\abs{g_n(x)-(1-x)}<\varepsilon$ and therefore $g_n$ uniformly converges to $1-x$.
    
    Now we show that every point $x\in [0,1]$ is asymptotically periodic with the 2-cycle $0$ and $1$.
    Fix $x\in [0,1)$ and $\varepsilon>0$. There are $n_0,n_1\in \mathbb{N}$ such that $x<(n_0+1)/(n_0+2)$ and $\varepsilon>1/(n_1+2)$. Put $N=\max(n_0,n_1)$, then for every odd $n\geq N$ we get
    $$\abs{f_1^n(x)-0}\overset{(\ref{skladani})}{=}\abs{F_n(x)}=\frac{x}{n+1}<\frac{\frac{n+1}{n+2}}{n+1}=\frac{1}{n+2}<\varepsilon$$
    and for even $n\geq N$ we get
    $$\abs{f_1^n(x)-1}\overset{(\ref{skladani})}{=}\abs{F_n(x)-1}=\abs{-\frac{x}{n+1}+1-1}=\frac{x}{n+1}<\frac{1}{n+2}<\varepsilon.$$
Hence every point $x\in [0,1]$ is asymptotically periodic with the cycle $0$ and $1$. Therefore there is no fixed point which completes the proof.
\end{proof}
\begin{remark}
 Note that this theorem is also true for Definitions 2-4 (the only periodic points are 0 and 1).
\end{remark}
\noindent\textbf{Acknowledgements}

\noindent The research was supported by grant SGS/18/2016 from the Silesian University in Opava. Support of this institution is gratefully acknowledged. The author thanks his supervisor Professor Marta \v{S}tef\'{a}nkov\'{a} for valuable suggestions and comments.

\end{document}